\documentclass[10pt]{article}

\usepackage[utf8x]{inputenc}
\usepackage{amsfonts}
\usepackage{amsmath}
\usepackage{amssymb}
\usepackage{amsthm}
\usepackage{mathrsfs}
\usepackage{latexsym}
\usepackage{graphicx}
\usepackage{bm}
\usepackage{inputenc}
\usepackage{enumerate}
\usepackage{enumitem}

\usepackage{hyperref}
\usepackage{graphicx}

\DeclareMathOperator{\Nor}{Nor}

\DeclareMathOperator{\interior}{Int}
\DeclareMathOperator{\loc}{loc}

\swapnumbers

\allowdisplaybreaks

\newtheorem{Lemma}{Lemma}[section]
\newtheorem{Corollary}[Lemma]{Corollary}
\newtheorem{Theorem}[Lemma]{Theorem}

\theoremstyle{definition}

\newtheorem{Definition}[Lemma]{Definition}

\theoremstyle{remark}
\newtheorem{Remark}[Lemma]{Remark}

\newtheoremstyle{proof*}
{3pt}
{3pt}
{\rmfamily}
{}
{\bfseries}
{.}
{.5em}
{\thmnote{#3}}
\theoremstyle{proof*}
\newtheorem*{proof*}{}

\title{Distance functions with dense singular sets}
\author{Mario Santilli}

\begin{document}

\maketitle

\begin{abstract} 
	We characterize the denseness of the singular set of the distance function from a $ \mathcal{C}^1 $-hypersurface in terms of an inner ball condition and
we address the problem of the existence of viscosity solutions of the Eikonal equation whose singular set (i.e. set of non-differentiability points) is not no-where dense.
\end{abstract}

\section{Introduction}

The distance function $ \bm{\delta}_K $ from a closed subset $ K \subseteq \mathbf{R}^n $ is a viscosity solution of the Eikonal equation $ | \nabla u |^2 = 1 $ on $ \mathbf{R}^n \sim K $ and it plays a central role in the theory of Hamilton-Jacobi equations. The function $ \bm{\delta}_K $ is locally semiconcave on $ \mathbf{R}^n \sim K $ and it is continuously differentiable on $ \mathbf{R}^{n} \sim (K \cup \overline{\Sigma(K)}) $ with a locally Lipschitz gradient, where $ \Sigma(K) $ is the set of non-differentiability points of $ \bm{\delta}_K $. In view of these facts the topological and measure-theoretic properties of the sets $ \Sigma(K) $ and $ \overline{\Sigma(K)} $ have always been a central theme of research (see \cite{MR1695025}, \cite{MR1941909}, \cite{MR2094267}, \cite{MR2336304}, \cite{MR3038120}). The set $ \Sigma(K) $ can be covered, outside a set of $ \mathcal{H}^{n-1} $ measure zero, by the union of countably many $ \mathcal{C}^2 $ hypersurfaces (see \cite{MR536060}). Assuming at least that $K$ is a closed $ \mathcal{C}^2 $ hypersurface, the Lebesgue measure of $ \overline{\Sigma(K)} $ is zero and upper bounds on the Hausdorff dimension of the set $\overline{\Sigma(K)}$ are known (\cite{MR1695025}, \cite{MR1941909}, \cite{MR2094267}, \cite{MR2336304}); see also \cite{MR3568029} for the case of $ \mathcal{C}^{1,1} $ hypersurfaces that are almost $ \mathcal{C}^2 $. On the other hand a well known example of Mantegazza and Mennucci in \cite[pag.\ 10]{MR1941909} describes a convex body $ C $ with $ \mathcal{C}^{1,1} $-boundary such that $\overline{\Sigma(\partial C)}$ is a no-where dense subset of $ C $ with positive Lebesgue measure. This example raises the natural question to understand if (and under which hypothesis) the set $ \overline{\Sigma(K)} $ can have interior points. It is particularly interesting the case $ K = \partial C $, where $ C $ is a convex body with $ \mathcal{C}^{1,1} $ boundary, since if this example exists then one can construct by a well known procedure a viscosity solution of the Eikonal equation on all of $ \mathbf{R}^n $ whose singular set is not no-where dense. This question was addressed in \cite[Theorem 1, footnote pag.\ 520]{MR2398240}, which contains the assertion that every viscosity solution of the Eikonal equation on an open subset of $ \mathbf{R}^n $ must be differentiable outside a no-where dense set. Unfortunately the proof of this statement is invalid (see \cite{Erratum}) and, as we show in this paper, it turns out that the statement is actually not true.

In this note we aim to establish the existence of a counterexample to the aforementioned assertion in \cite[Theorem 1]{MR2398240} and to provide geometric conditions on $ \mathcal{C}^1 $-hypersurfaces $ K $ that ensures that $ \overline{\Sigma(K)} $ has non empty interior. Specifically we prove the following facts: 
\begin{enumerate}
	\item If $ \Omega $ is an open subset with $ \mathcal{C}^1 $ boundary then $ \Omega \sim \overline{\Sigma(\partial \Omega)} \neq \varnothing $ if and only if $ \Omega $ satisfies an inner uniform ball condition on some open subset of $ \partial \Omega $ (see \ref{inner ball}-\ref{theorem}).
\item If $ K $ is a closed and connected $ \mathcal{C}^1 $ hypersurface that is $ \mathcal{C}^2 $ unrectifiable, then  $ \overline{\Sigma(K)} = \mathbf{R}^n $ (see \ref{Th 2}).
\item For most of the convex bodies $C$ with $ \mathcal{C}^1 $ boundary (in the sense of Baire Category) the set $ \Sigma(\partial C) $ is dense in $C$ (see \ref{Th 1}).
\item There exists a convex body $C$ with $ \mathcal{C}^{1,1}$ boundary such that $ \overline{\Sigma(\partial C)} $ has interior points (see \ref{Cor 2}). 
\item There exists viscosity solutions of the Eikonal equation $|\nabla u|^2 =1 $ on all of $ \mathbf{R}^n $ that are not differentiable on a set that is \emph{not} no-where dense (see \ref{Cor 2}).
\end{enumerate}

{\small\textbf{Acknowledgements.} I wish to thank Professor Ludovic Rifford, who kindly points me out the flaw in the proof of \cite[Theorem 1]{MR2398240}}; see \cite{Erratum}.

\section{Inner ball condition and dense singular sets}
In this section for an open set $ \Omega $ with $ \mathcal{C}^1 $ boundary $ K $ we characterize the denseness of the set of non differentiability points of $ \bm{\delta}_K $ in $ \Omega $ in terms of an inner ball condition (see \ref{theorem}). We use then this result to show that closed $ \mathcal{C}^1 $-hypersurfaces that are $ \mathcal{C}^2 $-unrectifiable have a singular set dense in all of $ \mathbf{R}^n $ (see \ref{Th 2}).

\begin{Definition}\label{submanifolds}
	Let $ k \geq 1 $ be an integer, $ 0 \leq \alpha \leq 1 $ and $ M \subseteq \mathbf{R}^n $. We say that $ M $ is a $ \mathcal{C}^{k, \alpha}$-hypersurface if and only if for every $ a \in M $ there exists an open subset $ U $ of $ \mathbf{R}^n $, an $ n-1 $ dimensional subspace $ Z $ of $ \mathbf{R}^n $ and a $ \mathcal{C}^{k, \alpha} $-diffeomorphism\footnote{This is a map $ \sigma \in \mathcal{C}^{k , \alpha}(U , \mathbf{R}^n) $ such that $ \sigma(U) $ is an open subset of $ \mathbf{R}^n $ and $ \sigma^{-1} \in \mathcal{C}^{k, \alpha}(\sigma(U), \mathbf{R}^n) $.} $ \sigma : U \rightarrow \mathbf{R}^n $ such that 
	\begin{equation*}
	\sigma(U \cap M) = Z \cap \sigma(U).
	\end{equation*}
\end{Definition}

\begin{Definition}
	Let $ k \geq 1 $ be an integer. A $ k $-manifold is an Hausdorff space which is locally homeomorphic to an open subset of $ \mathbf{R}^k $.
\end{Definition}

Let $ K \subseteq \mathbf{R}^n $ be a closed set and let $ \bm{\xi}_K: \mathbf{R}^n \rightarrow \bm{2}^{K} $ be the nearest point projection onto $ K $:
\begin{equation*}
\bm{\xi}_K(x) = K  \cap \big\{a: | x - a| = \bm{\delta}_K(x) \big\}
\end{equation*}
for every $ x \in \mathbf{R}^n $. The singular set of $ \bm{\delta}_K $ is defined as
\begin{equation*}
	\Sigma(K) = (\mathbf{R}^n \sim K) \cap \{ x : \textrm{$ \bm{\delta}_K $ is not differentiable at $ x $}   \}.
\end{equation*}
\begin{Remark}
The reader might wonder what are the points in $ K $ where $ \bm{\delta}_K $ is not differentiable. In this regard one observes that if $ x \in K $ and $ \bm{\delta}_K $ is differentiable at $ x $ then $ \nabla \bm{\delta}_K(x) =0 $. It follows that if $ x \in K $ and the tangent cone (see \cite[4.3]{MR0110078}) of $ K $ at $ x $ is not equal to $ \mathbf{R}^n $ then $ \bm{\delta}_K $ is not differentiable at $ x $. In particular if $ K $ is $ \mathcal{C}^1 $ hypersurface [resp.\ $K$ is a convex body] $ \bm{\delta}_K $ is not differentiable at all points of $ K $ [resp.\ all points of $ \partial K $].
\end{Remark}

\begin{Remark}
	It is well known that $ \bm{\delta}_K $ is locally semiconcave in $ \mathbf{R}^n \sim K $, see \cite[2.2.2]{MR2041617}. As a consequence of general structural results on the singular sets of convex functions (see \cite{MR536060}) we deduce that $ \Sigma(K) $ can be covered, outside a set of $ \mathcal{H}^{n-1} $ measure zero, by the union of countably many $ \mathcal{C}^2 $-hypersurfaces. 
\end{Remark}

We recall from \cite[3.4.5]{MR2041617} a well known characterization of $ \Sigma(K) $.

\begin{Lemma}\label{aux remark}
	Suppose $ K  \subseteq \mathbf{R}^n $ is closed and $ x \notin K $.
	
	Then $ x \notin \Sigma(K) $ if and only if $ \bm{\xi}_K(x) $ is a singleton and
	\begin{equation*}
	\nabla \bm{\delta}_{K}(x)  = \frac{x- \bm{\xi}_{K}(x)}{\bm{\delta}_{K}(x)}.
	\end{equation*}
\end{Lemma}

\begin{Definition}
	For $ x \in \mathbf{R}^n $ and $ r > 0 $ we define
	\begin{equation*}
	\mathbf{U}(x, r) = \mathbf{R}^n \cap \{ y : |y-x| < r  \}.
	\end{equation*}
\end{Definition}

\begin{Definition}\label{inner ball}
	Suppose $ \Omega $ is an open subset of $ \mathbf{R}^n $ and $ S \subseteq \partial \Omega $. We say that $ \Omega $ satisfies an \emph{inner uniform ball condition on $ S $} if and only if there exists $ \rho > 0 $ such that each $ x \in S $ belongs to the boundary of an open ball $ B $ of radius $ \rho $ which is contained in $ \Omega $.
\end{Definition}

\begin{Theorem}\label{theorem}
Let $ K \subseteq \mathbf{R}^n $ be a closed  $ \mathcal{C}^1 $-hypersurface and let $ \Omega $ be an open subset of $ \mathbf{R}^n $ such that $ \partial \Omega = K $.

Then $ \Omega \sim \overline{\Sigma(K)} \neq \varnothing $ if and only if $ \Omega $ satisfies an inner uniform ball condition on a non-empty open subset of $ K $.
\end{Theorem}

\begin{proof}
Suppose $ \Omega \sim \overline{\Sigma(K)} \neq \varnothing $. Choose $ w \in \Omega \sim \overline{\Sigma(K)} $ and $ 0 < \epsilon < \bm{\delta}_{K}(w) $ such that $ \mathbf{U}(w, \epsilon)  \subseteq \Omega \sim \overline{\Sigma(K)} $. Then define
	\begin{equation*}
		S = \mathbf{U}(w, \epsilon)  \cap \{ x : \bm{\delta}_{K}(x) = \bm{\delta}_{K}(w)  \}.
	\end{equation*}
Since, by \ref{aux remark}, the Lipschitz function $ \bm{\delta}_{K} $ is differentiable at each $ x \in \mathbf{U}(w, \epsilon)$ and $|\nabla \bm{\delta}_{K}(x)| = 1 $, we apply the implicit function theorem of Clarke \cite[7.11]{MR709590} to conclude that $ S $ is an $(n-1)$-manifold. Moreover $ \bm{\xi}_{K}| S $ is continuous by \cite[4.8(4)]{MR0110078}. We prove that $ \bm{\xi}_{K}|S $ is an injective map. Suppose $  x, y \in S $ such that $\bm{\xi}_{K}(x) = \bm{\xi}_{K}(y) $. Then
\begin{equation*}
x- \bm{\xi}_{K}(x) \in \Nor(K,\bm{\xi}_{K}(x) ), \quad y- \bm{\xi}_{K}(x) \in \Nor(K,\bm{\xi}_{K}(x) )
\end{equation*}
and $|x- \bm{\xi}_{K}(x)| = |y- \bm{\xi}_{K}(x)| = \bm{\delta}_{K}(w)$. Since $ \dim \Nor(K,\bm{\xi}_{K}(x) ) = 1 $, it follows that either $x- \bm{\xi}_{K}(x) =  y- \bm{\xi}_{K}(x) $ or $x- \bm{\xi}_{K}(x) =  \bm{\xi}_{K}(x) - y $. The latter would imply that 
\begin{equation*}
	| x-y | = |x- \bm{\xi}_{K}(x) + \bm{\xi}_{K}(x) - y| = 2 |x-\bm{\xi}_{K}(x)| = 2\bm{\delta}_{K}(w)
\end{equation*}
which is clearly impossible, since $|x-y| < 2\epsilon < 2\bm{\delta}_{K}(w) $. Henceforth $ x= y $ and $ \bm{\xi}_{K}|S $ is injective. Since $ S $ and $ \partial \Omega $ are $(n-1)$-manifolds, we apply Brouwer's theorem on invariance of domain (see \cite[IV,\ 7.4]{MR0415602}) to conclude that $\bm{\xi}_{K}(S) $ is open in $ K $. Noting that
\begin{equation*}
	 \mathbf{U}(x, \bm{\delta}_K(w)) \subseteq \Omega \quad \textrm{and} \quad \bm{\xi}_K(x) \in \partial \mathbf{U}(x, \bm{\delta}_K(w)),
\end{equation*}
for every $ x \in S $, we conclude that $ \Omega $ satisfies an inner uniform ball condition on $ \bm{\xi}_K(S) $.

Suppose $ S \subseteq K $ is open in $ K $ and $ \Omega $ satisfies an inner uniform ball condition on $ S $. Let $ \nu : K \rightarrow \mathbf{S}^{n-1} $ be the inner unit normal of $ \Omega $. Our hypothesis implies that there exists $ \rho > 0 $ such that 
\begin{equation*}
\mathbf{U}(a+\rho\nu(a), \rho) \subseteq \Omega \quad \textrm{for every $ a \in S $.}
\end{equation*}
Define $ \phi : S \times (0, \rho) \rightarrow \mathbf{R}^n $ by $ \phi(a,t) = a + t \nu(a) $ for $(a,t) \in S \times (0, \rho) $. Then $ \phi[S \times (0,\rho)] \subseteq \Omega $. If we prove that $ \phi[S \times (0,\rho)] $ is open in $ \mathbf{R}^n $ and $ \bm{\xi}_K(x) $ is a singleton for every $ x \in \phi[S \times (0,\rho)] $ then it is clear by \ref{aux remark} that $\phi[S \times (0,\rho)]$ does not intersect $ \overline{\Sigma(K)} $ and $ \Omega \sim \overline{\Sigma(K)} \neq \varnothing $. To prove the two assertions above we first show that $ \phi $ is injective. Let $(a,t), (b,s) \in S \times (0, \rho) $ such that $ a + t \nu(a) = b + s \nu(b) $. We notice that 
\begin{equation*}
t = \bm{\delta}_K(a+t\nu(a)) = \bm{\delta}_K(b+s\nu(b)) = s,
\end{equation*}
\begin{equation*}
|a+t\nu(a)-b| = t.
\end{equation*}
If $ a \neq b $ then $ |a+\rho\nu(a) - b| < \rho $ and $ b \in \Omega $, which is a contradiction. Henceforth $ a= b $ and $ \phi $ is injective. If $ b \in \bm{\xi}_K(a + t \nu(a)) $ for some $(a,t) \in S \times (0, \rho) $ then we notice that $ \nu(b) = t^{-1}(a + t\nu(a)-b) $ and $ a = b $ by the injectivity of $ \phi $. Therefore $ \bm{\xi}_K(x) $ is a singleton for every $ x \in \phi[S \times (0, \rho)] $. Moreover, since $ S \times (0, \rho) $ is an $n$-manifold, we conclude that $ \phi[S \times (0, \rho)] $ is an open subset of $ \mathbf{R}^n $ by \cite[IV,\ 7.4]{MR0415602}.
\end{proof}

This corollary shows that every $ \mathcal{C}^1 $-hypersurface that is $ \mathcal{C}^2 $-unrectifiable generates a dense singular set. 

\begin{Corollary}\label{Th 2}
Suppose $ K $ is a closed and connected $ \mathcal{C}^1 $ hypersurface such that $ \mathcal{H}^{n-1}(K \cap M) = 0 $ whenever $ M $ is a $ \mathcal{C}^2 $-hypersurface of $ \mathbf{R}^n $.

Then $ \overline{\Sigma(K)} = \mathbf{R}^n $.
\end{Corollary}

\begin{proof}
 Let $ U $ and $ V $ the two connected open subsets of $ \mathbf{R}^n $ such that $ \partial U = \partial V = K $, $ U \cap V = \varnothing $ and $ U \cup V \cup K = \mathbf{R}^n $. It follows from \cite{MR4012808} that if $ S $ is a subset of $K$ such that either $ U $ or $ V $ satisfies an inner uniform ball condition on $ S $ then $ \mathcal{H}^{n-1}(S) =0 $. In particular neither $ U $ nor $ V $ can satisfy an inner uniform ball condition on some non empty open subset of $ K $. Therefore we conclude from \ref{theorem} that $ \overline{\Sigma(K)} = \mathbf{R}^n $.
\end{proof}

\begin{Remark}\label{aux remark 1}
	Let $ 0 < \alpha < 1 $. It follows from \cite{MR0427559} that there exists a function $ f : \mathbf{R}^{n-1} \rightarrow \mathbf{R} $ whose graph $K$ is  a closed $ \mathcal{C}^{1, \alpha} $-hypersurface such that 	$ \mathcal{H}^{n-1}(K \cap M) =0 $ for every $ \mathcal{C}^2 $-hypersurface $ M \subseteq \mathbf{R}^n $.
\end{Remark}

\begin{Remark}
It follows from the theory of sets of positive reach (see \cite[\S 4]{MR0110078}) that if $ K $ is a closed $ \mathcal{C}^{1,1} $-hypersuface then there exists an open neighbourhood $ U $ of $ K $  such that $ \overline{\Sigma(K)} \cap U = \varnothing $. 
\end{Remark}

\section{Convex sets}

In this section we show that there exist many $ \mathcal{C}^{1,1} $ convex hypersurfaces $ K $ such that $ \overline{\Sigma(K)} $ has non empty interior. Consequently there exist many viscosity solutions of the Eikonal equation on $ \mathbf{R}^n $ such that the singular set is not no-where dense.

Let $ \mathcal{K}_r^n $ be the space of all compact convex subsets in $ \mathbf{R}^n $ with non empty interior such that $ \partial C $ is a $ \mathcal{C}^1 $  hypersurface. We equip  $ \mathcal{K}^n_r $ with the Hausdorff metric and we recall (see \cite[2.7.1]{MR3155183}) that it is a Baire space\footnote{In fact $ \mathcal{K}^n_r $ is a comeager of the space of all convex bodies (with non empty interior) equipped with the Haussdorf metric.} (i.e.\ countable intersections of dense open subsets are dense). A subset of a metric space is called \emph{meager} if and only if it is countable union of nowhere-dense sets and it is called \emph{comeager} if and only if it is the complementary of a meager set. It is customary to call \emph{typical} the elements of a comeager subset of a Baire space.

The next statement contains the observation that for a typical convex body $C \in \mathcal{K}^n_r$ the distance function from the boundary $ \partial C $ is not differentiable on a dense subset of $ C $. This statement easily follows combining Theorem \ref{theorem} with well known properties of the curvature of a typical convex body.

\begin{Theorem}\label{Th 1}
	For all $ C $ in $ \mathcal{K}^n_r $, except those belonging to a meager subset of $ \mathcal{K}^n_r $, 
	\begin{equation*}
	 C = \overline{\Sigma(\partial C)}.
	\end{equation*}
\end{Theorem}

\begin{proof}
By \cite[2.7.4]{MR3155183} there exists a comeager $ \mathcal{T} $ of $ \mathcal{K}^n_r $ such that if $ C \in \mathcal{T} $ then $ \interior(C) $ does not satisfy an inner uniform ball condition on a comeager subset of $ \partial C $. It follows from \ref{theorem} that $ C \subseteq \overline{\Sigma(\partial C)} $ for every $ C \in \mathcal{T} $.  On the other hand it is well known that $ \bm{\delta}_{\partial A} \in \mathcal{C}^{1,1}_{\loc}(\mathbf{R}^n \sim A) $ for every convex body $A$ (see for instance \cite[4.8]{MR0110078}) and the conclusion follows.
\end{proof}

\begin{Lemma}\label{Lem 1}
If $ C $ is a convex body then for every $ \epsilon > 0 $ the set 
\begin{equation*}
C_\epsilon = \mathbf{R}^n \cap \{ x : \bm{\delta}_C(x) \leq \epsilon    \}
\end{equation*}
is convex, $ \partial C_\epsilon $ is a $ \mathcal{C}^{1,1} $-hypersurface and $  \Sigma(\partial C) \subseteq \Sigma(\partial C_\epsilon) $.
\end{Lemma}

\begin{proof}Evidently $ C_\epsilon $ is a convex body and is well known that $ \partial C_\epsilon $ is a $ \mathcal{C}^{1,1} $ hypersurface (see \cite[4.8]{MR0110078}). We observe that
\begin{equation*}
\bm{\delta}_{\partial C_\epsilon}(x) = \epsilon + \bm{\delta}_{\partial C}(x) \quad \textrm{for $ x \in C $,}
	\end{equation*} 
	and we conclude that $ \Sigma(\partial C) \subseteq \Sigma(\partial C_\epsilon) $.
\end{proof}

\begin{Theorem}\label{Cor 2}
There exists $ C \in \mathcal{K}_r^n $ such that $ \partial C $ is a $ \mathcal{C}^{1,1} $-hypersurface and $ \overline{\Sigma(\partial C)} $ has non empty interior. Moreover the function $ u : \mathbf{R}^n \rightarrow \mathbf{R} $ defined by
\begin{equation*}
 u(x) = \bm{\delta}_{\partial C}(x) \quad \textrm{for $ x \in C $} \qquad \textrm{and} \qquad  u(x) = -\bm{\delta}_{\partial C}(x) \quad \textrm{for $ x \in \mathbf{R}^n \sim C $}
\end{equation*}
is a viscosity solution of the Eikonal equation on $ \mathbf{R}^n $ and the closure of the set of points where $ u $ is not differentiable has non empty interior.
\end{Theorem}

\begin{proof}
The existence of a convex body $ C $ such that $ \partial C $ is a $ \mathcal{C}^{1,1} $ hypersurface and $ \overline{\Sigma(\partial C)} $ has non empty interior directly follows from \ref{Th 1} and \ref{Lem 1}.

It follows from \cite[4.20]{MR0110078} that $ \partial C $ has positive reach. Therefore one infers from \cite[Theorem 2]{MR614221} that there exists an open neighborhood $ U $ of $ \partial C $ such that $ u $ is continuously differentiable on $ U $. Since it is clear that $ u $ is continuously differentiable on $ \mathbf{R}^n \sim C $, we conclude by \ref{aux remark} that
\begin{equation*}
	| \nabla u(x) |^2 = 1 \quad \textrm{for every $ x \in (\mathbf{R}^n \sim C) \cup U $.}
\end{equation*}
Moreover $ u $ is locally semiconcave on the interior $ \interior(C) $ of $ C $ and $ |\nabla u(x)|^2  = 1 $ for $ \mathcal{L}^n $ a.e.\ $ x \in \interior(C) $. Henceforth, it follows from \cite[5.3.1]{MR2041617} that $ | \nabla u |^2 = 1 $ in the viscosity sense in $ \interior(C) $. It is now evident that $ | \nabla u |^2 = 1 $ in the viscosity sense in $ \mathbf{R}^n $.
\end{proof}



\medskip 

\noindent Institut f\"ur Mathematik, Universit\"at Augsburg, \newline Universit\"atsstr.\ 14, 86159, Augsburg, Germany,
\newline mario.santilli@math.uni-augsburg.de

\end{document}